\documentclass[11pt]{article}
\usepackage{amsmath, amsthm, amssymb,lineno}
\numberwithin{equation}{section}
\newtheorem{theorem}{Theorem}

\begin{document}
\author{Ajai Choudhry}
\title{Expressing an integer as a sum of\\ cubes of  polynomials}
\date{}
\maketitle

\begin{abstract}
In this paper we prove that there exist infinitely many integers which can be expressed as a sum of four cubes of polynomials with integer coefficients. We give several identities that express the integers 1 and 2 as a sum of four cubes of polynomials. We also show that every integer can be expressed as a sum of five cubes of  polynomials with integer coefficients. 
\end{abstract}

\section{Introduction}\label{intro}
This paper is concerned with the problem of expressing an integer  as a sum of cubes of polynomials  with integer coefficients. As mentioned by Mordell \cite{Mo}, the first such result was obtained in 1908 by Werebrusow  who proved the identity,
\[ (1+6t^3)^3+(1-6t^3)^3+(-6t^2)^3=2. \]
Subsequently, in 1936, Mahler \cite{Ma} proved the identity,
\[ (9t^4)^3+(3t-9t^4)^3+(1-9t^3)^3 =1. \] 
If $a$ is any arbitrary integer,  we can multiply these identities by $a^3$ to express the integers $2a^3$ and $a^3$, respectively, as a sum of three cubes of  polynomials, but it is not known whether there is any other integer that can be expressed as a sum of three cubes of  polynomials with integer coefficients. Identities expressing an integer as a sum of four or five cubes of polynomials have apparently not been published.

We show in this paper that there exist infinitely many integers that are expressible as a sum of four cubes of polynomials with integer coefficients. Specifically, we give several identities expressing the integers 1 and 2 as a sum of four cubes of  polynomials. We also prove that every integer is expressible as a sum of five cubes of  polynomials with integer coefficients.

\section{Integers expressed as sums of four cubes of polynomials}

\subsection{Two polynomial identities}\label{identpq}
The theorem below gives  identities that describe infinitely many integers which can written as a sum of four cubes of polynomials.
\begin{theorem}
If $p$ and $q$ are arbitrary integers, any integer expressible either as $p^3+q^3$ or $2(p^6-q^6)$ is expressible as a sum of four cubes of univariate polynomials with integer coefficients.
\end{theorem}
\begin{proof} The proof is based on the following two identities in both of which $t$ is an arbitrary parameter:
\begin{multline}
p^3+q^3 = \{2(p + q)t^2 + 4qt + q)\}^3+\{2(p + q)t^2 + 4qt - p + 2q\}^3\\
+\{-2(p + q)t^2 + (p - 3q)t + p\}^3+\{-2(p +q)t^2 - (p + 5q)t + p - 2q\}^3, \label{identpq1}
\end{multline}
and
\begin{equation}
2(p^6-q^6) =(pt-q^2)^3+(-pt-q^2)^3 +(qt+p^2)^3+(-qt+p^2)^3. \label{identpq2}
\end{equation}

 To obtain the identity \eqref{identpq1}, we  solve the equation
\begin{equation}
x_1^3+x_2^3+x_3^3+x_4^3=p^3+q^3, \label{eqnpq}
\end{equation}
by writing,
\begin{equation}
x_1=-y+p,x_2=-y+q, x_3=y+m,x_4=y-m,  \label{subspq}
\end{equation}
when \eqref{eqnpq} is readily solved for $y$  to yield a nonzero solution given by $y=-(2m^2 - p^2 - q^2)/(p + q)$.  On writing $m=(p+q)t+q$, we get $y=-2(p+q)t^2 - 4qt + p - q$, and on substituting the values of $m$ and $y$ in \eqref{subspq}, we get the values of $x_i, i=1, \ldots, 4$, and  we thus get the identity \eqref{identpq1}.

To prove the identity \eqref{identpq2}, we consider
\[(pt+u)^3+(-pt+u)^3+(qt+v)^3+(-qt+v)^3   \]
 as a cubic polynomial in $t$. It is evident that the coefficients of $t^3$ and $t$ in this polynomial vanish. We now choose $u=-q^2,v=p^2$, so that  the coefficient of $t^2$ also vanishes, and the polynomial reduces to $2(p^6-q^6)$, and we thus get the identity \eqref{identpq2}.

It immediately follows from the identities \eqref{identpq1} and \eqref{identpq2} that integers that may be written as $p^3+q^3$ or $2(p^6-q^6)$ can be written as a sum of four cubes of univariate polynomials with integer coefficients.
\end{proof}

As numerical applications of the identity \eqref{identpq1}, we give two examples by taking $(p,q)=(2, -1)$ and $(p,q)=(2, 1)$ when we obtain the following two identities expressing the integers 7 and 9, respectively, as a sum of four cubes of  polynomials with integer coefficients:
\begin{equation*}
\begin{aligned}
(2t^2 - 4t - 1)^3+(2t^2 - 4t - 4)^3+(-2t^2 + 5t + 2)^3+(-2t^2 + 3t + 4)^3 & =7, \\
(6t^2 + 4t + 1)^3 + (6t^2 + 4t)^3 + (-6t^2 - t + 2)^3 + (-6t^2 - 7t)^3  & =9.
\end{aligned}
\end{equation*}

Similarly, on taking $(p, q)=(2,1)$ in the identity \eqref{identpq2}, we get the following identity which expresses the integer 126 as a sum of four cubes of  polynomials: 
\begin{equation*}
(2t - 1)^3 + (-2t - 1)^3 + (t + 4)^3 + (-t + 4)^3 = 126.
\end{equation*}

\subsection{Expressing 1 as a sum of four cubes of  polynomials}\label{identone} 
We will now give several identities that express 1 as a sum of four cubes of polynomials.

To obtain the first identity, we will solve the equation
\begin{equation}
x_1^3+x_2^3+x_3^3+x_4^3=1, \label{eqone}
\end{equation}
by writing
\begin{equation}
x_1=py+1, x_2=-py+m, x_3=-py-m, x_4=py, \label{subsone}
\end{equation}
when Eq. \eqref{eqone} is readily solved for $y$ to obtain a nonzero solution given by $y=(2m^2 - 1)/p$, and on writing $m=pm_1+m_2$, we get 
\[
y=2pm_1^2 + 4m_1m_2 + (2m_2^2 - 1)/p.
\]
Now on taking $p=2m_2^2 - 1$, we get 
\[
y= 4m_1^2m_2^2 - 2m_1^2 + 4m_1m_2+1,
\]
and on substituting the values of $m, p$ and $y$ in \eqref{subsone}, we get the following solution of Eq. \eqref{eqone}:
\begin{equation}
\begin{aligned}
x_1 & = 2(2m_2^2 - 1)^2m_1^2 + 4m_2(2m_2^2 - 1)m_1 + 2m_2^2, \\
x_2 & = -2(2m_2^2 - 1)^2m_1^2 - (4m_2 - 1)(2m_2^2 - 1)m_1 - (2m_2 + 1)(m_2 - 1), \\
x_3 & = -2(2m_2^2 - 1)^2m_1^2 - (4m_2 + 1)(2m_2^2 - 1)m_1 - (m_2 + 1)(2m_2 - 1), \\
x_4  & = 2(2m_2^2 - 1)^2m_1^2 + 4m_2(2m_2^2 - 1)m_1 + 2m_2^2 - 1.
\end{aligned}
\end{equation}

Since the values of $x_i, i=1, \ldots, 4$, are  given by polynomials, with integer coefficents, in two variables, we have obtained an identity that expresses  the integer 1 as a sum of four cubes of polynomials in two variables.

We give below three  more identities expressing 1 as a sum of four cubes of   univariate polynomials:
\begin{equation*}
(2t^2)^3+(2t^2-1)^3+(-2t^2 - t + 1)^3+(-2t^2 + t + 1)^3=1, 
\end{equation*}
and 
\begin{multline*}
(8t^3 - 2t^2 - 4t + 1)^3 +  (8t^3 - 6t^2 - 3t + 2)^3\\
+(-8t^3 + 2t^2 + 3t)^3 + (-8t^3 + 6t^2 + 4t - 2)^3 =1, 
\end{multline*}
and
\begin{equation*}
(3t^6 + 3t^3 + 1)^3 + \{-3t^3(t^3 + 1)\}^3 + (-3t^4 - 2t)^3 +(-t)^3 = 1. 
\end{equation*}

These identities may be readily verified by direct computation.

\subsection{Expressing 2 as a sum of four cubes of polynomials}\label{identtwo}
The following identity  expresses the integer 2 as a sum of of four cubes of polynomials in three variables:
\begin{multline}
\{6t^3(g^3 + h^3)^2 + 1\}^3+ \{-6t^3(g^3 + h^3)^2 + 1\}^3\\
+\{-6gt^2(g^3 + h^3)\}^3+\{-6ht^2(g^3 + h^3)\}^3=2. \label{identint2first}
\end{multline}

To obtain the identity \eqref{identint2first} we begin by writing
\begin{equation}
(p+1)^3+(-p+1)^3+q^3+r^3=2 \label{eq2}
\end{equation}
which reduces to 
\begin{equation}
6p^2 + q^3+r^3=0. \label{eq2a}
\end{equation}

We solve Eq. \eqref{eq2a} by writing $p=fm,q=gm,r=hm$, when we get,
\begin{equation}
m=-6f^2/(g^3 + h^3).
\end{equation}
To obtain a solution with  integer coefficients, we write $f=t(g^3+h^3)$, when we get $ m=-6t^2(g^3 + h^3)$ which leads to the solution,
\begin{equation*}
p = -6t^3(g^3 + h^3)^2, \quad q = -6gt^2(g^3 + h^3), \quad r = -6ht^2(g^3 + h^3),
\end{equation*}
and on substituting these values in \eqref{eq2}, we get the identity \eqref{identint2first}.

We give below three more identities expressing the integer 2 as a sum of four cubes of univariate polynomials:
\begin{align*}
 (t^2)^3+(t^2)^3+(-t^2 + t + 1)^3+(-t^2 - t + 1)^3&=2,\\
(3t^3 + 1)^3 +(-3t^3 + 1)^3  +(-3t^2)^3 +(-3t^2)^3& =2,\\
(18t^3 + 1)^3+(-18t^3 + 1)^3  +(-6t^2)^3 + (-12t^2)^3&=2.
\end{align*}
These identities may be readily verified by direct computation. More such identities can be obtained.

\section{Integers expressed as sums of five cubes of polynomials}
We will now prove that every integer can be expressed as a sum of five cubes of polynomials.
\begin{theorem} Every integer is expressible as a sum of five cubes of univariate polynomials with integer coefficients.
\end{theorem}
\begin{proof}
The proof is based on the following six identities:
\begin{equation}
6m   =(36t^3 + m + 1)^3 + (36t^3 + m - 1)^3 + 2(-36t^3 - m)^3 + (-6t)^3,\label{5cubes0}
\end{equation}
\begin{multline}
6m+1 = (36t^3 - 18t^2 + 3t +m + 1)^3 + (36t^3 - 18t^2  + 3t +m - 1)^3 \\
       + 2(-36t^3 + 18t^2 - 3t - m)^3 + (-6t + 1)^3, \label{5cubes1}
\end{multline}
\begin{multline}
6m+2   =(36t^3 - 36t^2 + 12t+m )^3 + (36t^3 - 36t^2  + 12t+m  - 2)^3 \\
  + 2(-36t^3 + 36t^2 - 12t - m + 1)^3 + (-6t + 2)^3,\label{5cubes2}
	\end{multline}
\begin{multline}
6m+3  =(36t^3 - 54t^2  + 27t+m - 3)^3 + (36t^3 - 54t^2  + 27t+m  - 5)^3 \\
 + 2(-36t^3 + 54t^2 - 27t - m + 4)^3 + (-6t + 3)^3,\label{5cubes3}
\end{multline}
\begin{multline}
6m+4  =(36t^3 + 36t^2  + 12t +m + 3)^3 + (36t^3 + 36t^2  + 12t +m + 1)^3 \\
  + 2(-36t^3 - 36t^2 - 12t - m - 2)^3 + (-6t - 2)^3,\label{5cubes4}
	\end{multline}
\begin{multline}
6m+5  =(36t^3 + 18t^2  + 3t+m  + 2)^3 + (36t^3 + 18t^2  + 3t+m )^3 \\
 + 2(-36t^3 - 18t^2 - 3t - m - 1)^3 + (-6t - 1)^3,\label{5cubes5}
\end{multline}
where $m$ and $t$ are arbitrary parameters.

To obtain the above identities, we begin with the following simple, readily verifiable, identity:
\begin{equation}
6r= (r + 1)^3 + (r - 1)^3 - 2r^3. \label{identsimple1}
\end{equation}
On adding $(-6t+j)^3$ to both sides, we get the identity
\begin{equation}
6r+(-6t+j)^3= (r + 1)^3 + (r - 1)^3 - 2r^3+ (-6t+j)^3.\label{identsimple2}
\end{equation}

We now equate the left-hand side of the identity \eqref{identsimple2} to $6m+j$, where $j$ is an integer, and solve for $r$. Since $j^3 \equiv j \mod 6$, we get a value of $r$, say $r=r_0$, which is given by a polynomial, with integer coefficients, in the parameters $m$ and $t$. On replacing $r$ by $r_0$ in the identity \eqref{identsimple2}, we get an identity that expresses $6m+j$ as a sum of five cubes of polynomials with integer coefficients. By successively taking $j=0, 1, 2, 3$, we   obtained the identities \eqref{5cubes0}--\eqref{5cubes3}. 

While we can obtain identities expressing $6m+4$ and $6m+5$ as a sum of five cubes of polynomials in the same way, we obtained the simpler identity \eqref{5cubes4} by adding $(-6t-2)^3$ (instead of $(-6t+4)^3$) to both sides of \eqref{identsimple1} and proceeding as above, while for the  identity \eqref{5cubes5}, we added $(-6t-1)^3$ (instead of $(-6t+5)^3$) to both sides of \eqref{identsimple1} and followed the same procedure. The identities \eqref{5cubes0}--\eqref{5cubes5} can  also be readily verified by direct computation.

Since any arbitrary integer $n$ is expressible as $6m+j$ where $ j \in \{0,1,\ldots$, $5\}$, it follows from the identities \eqref{5cubes0}--\eqref{5cubes5} that every integer is expressible as the sum of five cubes of univariate polynomials with integer coefficients.
\end{proof}

As numerical examples, taking $m=0$ in the identities \eqref{5cubes3} and \eqref{5cubes4}, we get the following two identities expressing the integers 3 and 4, respectively, as a sum of five cubes of  polynomials:
\begin{align*}
(36t^3 - 54t^2  + 27t - 3)^3 + (36t^3 - 54t^2  + 27t - 5)^3& \\
 + 2(-36t^3 + 54t^2 - 27t + 4)^3 + (-6t + 3)^3  & =3, \\
(36t^3 + 36t^2  + 12t + 3)^3 + (36t^3 + 36t^2 + 12t + 1)^3 & \\
  + 2(-36t^3 - 36t^2 - 12t  - 2)^3 + (-6t - 2)^3& = 4.
\end{align*}

\section{Some open problems}
It would be interesting to find identities expressing  small integers such as 3 and 4 as a sum of four cubes of  polynomials. In fact, it would be useful to determine which integers can be expressed as a sum of four cubes of  polynomials and for which integers it becomes necessary to use five cubes of polynomials. 

It is a far more  challenging  problem to find new  identities  expressing an integer  as a sum of three cubes of  polynomials with integer coefficients.    In fact, it seems unlikely that, apart from  integers of the type $a^3$ and $2a^3$ mentioned in the Introduction,  there are any other integers that can  be  expressed as a sum of three cubes of  polynomials.

\noindent Postal address: Ajai Choudhry, 13/4 A, Clay Square, Lucknow - 226001, India\\

\noindent E-mail: ajaic203@yahoo.com

\end{document}